\documentclass[12pt,a4paper,twoside,reqno]{amsart}
\usepackage[british,UKenglish,USenglish,english,american]{babel}

\usepackage{tikz}
\usetikzlibrary{arrows}

\usepackage{amsthm,amssymb,amsmath} 

\usepackage[utf8]{inputenc}

\usepackage{enumerate}
\usepackage[all]{xy}

\usepackage{mathtools}

\makeatletter
\def\@map#1#2[#3]{\mbox{$#1 \colon\thinspace #2 \longrightarrow #3$}}
\def\map#1#2{\@ifnextchar [{\@map{#1}{#2}}{\@map{#1}{#2}[#2]}}
\g@addto@macro\@floatboxreset\centering
\makeatother

\renewcommand{\epsilon}{\ensuremath{\varepsilon}}
\renewcommand{\phi}{\ensuremath{\varphi}}

\renewcommand{\to}{\ensuremath{\longrightarrow}}
\renewcommand{\mapsto}{\ensuremath{\longmapsto}}

\newcommand{\Z}{\ensuremath{\mathbb{Z}}}

\newcommand{\ang}[1]{\ensuremath{\left\langle #1\right\rangle}}

\newtheoremstyle{theoremm}{}{}{\itshape}{}{\scshape}{.}{ }{}
\theoremstyle{theoremm}
\newtheorem{thm}{Theorem}[section]
\newtheorem{lem}[thm]{Lemma}
\newtheorem{prop}[thm]{Proposition}
\newtheorem{cor}[thm]{Corollary}

\newtheoremstyle{thmintro}{}{}{\itshape}{}{\scshape}{.}{ }{}
\theoremstyle{thmintro}
\newtheorem{theoremA}{Theorem}

\newtheoremstyle{remark}{}{}{}{}{\scshape}{.}{ }{}

\theoremstyle{remark}
\newtheorem{defn}[thm]{Definition}
\newtheorem{rem}[thm]{Remark}

\newcommand{\aut}[1]{\ensuremath{\operatorname{\text{Aut}}\left({#1}\right)}}

\renewcommand{\ker}[1]{\ensuremath{\operatorname{\text{Ker}}\left({#1}\right)}}

\newcommand{\redefn}[1]{Definition~\protect\ref{defn:#1}}
\newcommand{\rethm}[1]{Theorem~\protect\ref{thm:#1}}
\newcommand{\relem}[1]{Lemma~\protect\ref{lem:#1}}
\newcommand{\reprop}[1]{Proposition~\protect\ref{prop:#1}}

\newcommand{\rerem}[1]{Remark~\protect\ref{rem:#1}}

\numberwithin{equation}{section}
\allowdisplaybreaks

\usepackage[pdftex,%
bookmarks=true,
breaklinks=true,
bookmarksnumbered = true,
colorlinks= true,
urlcolor= green,
anchorcolor = yellow,
citecolor=blue,
]{hyperref}                   

\begin{document}

\title{Universal virtual braid groups}

\author[Oscar Ocampo]{Oscar Ocampo}
\address{Universidade Federal da Bahia, Departamento de Matem\'atica - IME, CEP:~40170-110 - Salvador, Brazil}
\email{oscaro@ufba.br}



\date{\today}

\begin{abstract}

We introduce the universal virtual braid group $UV_n(c)$, which provides a unified algebraic framework for virtual braid--type structures with $c$ types of crossings and admits natural quotient maps onto the standard families in the literature. 
We prove that $UV_n(c)$ contains a right-angled Artin subgroup of finite index, yielding strong structural consequences: residual finiteness, linearity, solvability of the word and conjugacy problems, and the Tits alternative. For $n\ge 5$, the commutator subgroup $UV_n(c)'$ is perfect, and every non-abelian finite image contains a subgroup isomorphic to the symmetric group $S_n$; in particular, $S_n$ is the smallest non-abelian finite quotient. These rigidity phenomena persist under a broad class of natural quotients, including virtual braid, virtual singular braid, virtual twin and multi-virtual braid groups. 
We further obtain a complete classification of subgroup separability (LERF) and the Howson property for $UV_n(c)$ and its pure subgroup $PUV_n(c)$, showing that both properties hold precisely for $n\le 3$. We also compute the virtual cohomological dimension, determine the center, prove that the finite-index RAAG subgroup is characteristic, and construct explicit finite quotients of $UV_n(c)$ whose order is strictly larger than $n!$.

\end{abstract}

\maketitle

\section{Introduction}

Braid groups and their generalizations form a fundamental class of groups in geometric and combinatorial group theory.
Virtual braid groups were introduced by Kauffman~\cite{Kau} in the context of virtual knot theory and further studied by Bardakov~\cite{B}, Vershinin~\cite{Vershinin}, and others. 
Further extensions, including virtual singular braids~\cite{CPM} and multi--virtual braid groups~\cite{Kau2}, enrich the crossing structure while retaining a Coxeter--type virtual symmetry. 
Despite their common features, these families are typically treated as separate constructions, which suggests the existence of a universal algebraic structure from which they arise as natural quotients and whose intrinsic properties dictate their global algebraic behavior.

The definition of the universal virtual braid group \( UV_n(c) \) is motivated by two complementary observations. 
First, several families of virtual braid-type groups (such as the virtual braid groups \( VB_n \), the virtual singular braid groups \( VSG_n \), the virtual twin groups \( VT_n \), and the multi-virtual braid groups \( M_kVB_n \)) share algebraic properties whose proofs often follow similar patterns, yet are handled separately in the literature. By working with a universal group that admits each of these families as a quotient, one can unify such proofs: for properties that are preserved under taking quotients (e.g., residual finiteness, or the property that the commutator subgroup is perfect), establishing them at the level of \(UV_{n}(c)\) immediately yields the corresponding results for all quotients that retain the virtual generators intact.

Second, and more subtly, the universal group approach allows one to investigate \emph{which properties are inherited and which are not}. 
Studying a property \( \mathcal{P} \) in \( UV_n(c) \) and examining whether it passes to the quotients provides not only a unified proof but also a deeper understanding of the quotients themselves: properties that fail to be inherited reveal essential aspects of the interaction between crossing generators and virtual generators in each specific family. This duality between unification and discrimination is a central theme of the present work.

With this motivation in mind, we now introduce the \emph{universal virtual braid group} \( UV_n(c) \), defined for \( n \geq 2 \) strands and \( c \geq 1 \) types of (non-virtual) crossings (see Definition~\ref{defn:uvnc}). 
The group \( UV_n(c) \) admits canonical quotient maps onto the standard virtual braid, virtual singular, and multi-virtual braid-type groups (Proposition~\ref{prop:natural-quotients}), and therefore serves as a unifying structural object within the landscape of virtual braid-type constructions. 
Our first main result identifies a finite--index right--angled Artin subgroup.

\begin{theoremA}\label{thm:A}
For every $n\ge 2$ and $c\ge 1$, the group $UV_n(c)$ contains a right--angled Artin subgroup of finite index.
\end{theoremA}

More precisely, the subgroup $KUV_n(c)=\ker{\pi_n^K}$ (Definition~\ref{defn:puvkuv}) is a right--angled Artin group of index $n!$ in $UV_n(c)$ (Theorem~\ref{thm:preskuvnc} and Remark~\ref{rem:index}).
This structural decomposition places $UV_n(c)$ within the class of groups virtually of RAAG type (cf.~\cite{Davis2008}), yielding strong algebraic consequences such as residual finiteness and linearity over $\mathbb{Z}$ (Theorem~\ref{thm:virtually-raag}).
Moreover, it allows us to determine the virtual cohomological dimension of $UV_n(c)$, which we show to be $\lfloor n/2\rfloor$ (Theorem~\ref{thm:vcd}), and to establish that $UV_n(c)$ is Hopfian and has exponential growth (Theorem~\ref{thm:structural} and Theorem~\ref{thm:virtually-raag}).

A second structural mechanism emerges in the analysis of finite quotients and derived subgroups.

\begin{theoremA}\label{thm:B}
For $n\ge 5$, the commutator subgroup $UV_n(c)'$ is perfect. Furthermore, every non--abelian finite image of $UV_n(c)$ contains a subgroup isomorphic to the symmetric group $S_n$.
In particular, $S_n$ is the smallest non--abelian finite quotient of $UV_n(c)$.
\end{theoremA}

The perfectness of the commutator subgroup is established in Proposition~\ref{prop:UV-commutator-perfect}, while the statement about non--abelian finite quotients is proved in Proposition~\ref{prop:finite_image} and Theorem~\ref{thm:minimal_sn}.
Thus the virtual Coxeter component imposes global constraints on finite images: non--abelian quotients are already controlled by $S_n$, and this minimality phenomenon persists under a broad class of natural quotients (Proposition~\ref{prop:quotient_stable}), including the standard virtual braid, virtual singular, virtual twin and multi--virtual braid groups (Corollary~\ref{cor:examples}).

Using the RAAG structure of $KUV_n(c)$, we obtain a complete classification of subgroup separability (LERF) and the Howson property for $UV_n(c)$ and its pure subgroup $PUV_n(c)$. We show that $UV_n(c)$ is LERF if and only if $n\le 3$ (Theorem~\ref{thm:uv_lerf}), and likewise for the Howson property (Theorem~\ref{thm:uv_howson}). These results are sharp: for $n\ge 4$, the presence of a subgroup isomorphic to $F_2\times F_2$ in $KUV_n(c)$ (Corollary~\ref{cor:kuv_small}) obstructs both properties.

We further investigate the structure of finite quotients and homomorphisms to symmetric groups.
For $n \ge 5$, we classify all homomorphisms from $UV_n(c)$ to $S_m$ with $m \le n$, up to conjugation in $S_m$ (Theorem~\ref{thm:UV_to_Sm}).
As a consequence, every surjective homomorphism $UV_n(c) \to S_n$ is, up to conjugation, one of the maps $\phi_{(\epsilon_1,\dots,\epsilon_{c+1})}$ determined by the images of the generators (Corollary~\ref{cor:surj_to_Sn}).

Finally, we establish that the subgroup $KUV_n(c)$ is characteristic in $UV_n(c)$ for all $n\ge 2$ and $c\ge 1$ (Proposition~\ref{prop:characteristic-K}), and we construct explicit finite quotients of $UV_n(c)$ whose order is strictly larger than $n!$ (Corollary~\ref{cor:finite-quotients}), providing evidence that the symmetric group is indeed the smallest non--abelian finite quotient but not the only one.

Taken together, these results show that $UV_n(c)$ is a structurally rigid and conceptually unifying object, whose algebraic behavior is governed by the interaction between its Coxeter--type virtual part and its crossing generators.

The paper is organized as follows. Section~\ref{sec:uvnc} introduces $UV_n(c)$, describes its basic properties and natural quotients (Propositions~\ref{prop:natural-quotients} and~\ref{prop:semidirect}), and establishes the finite--index right--angled Artin subgroup $KUV_n(c)$ (Theorem~\ref{thm:preskuvnc}), thereby proving Theorem~\ref{thm:A}. 
Section~\ref{sec:structural} develops structural consequences, including the virtual cohomological dimension (Theorem~\ref{thm:vcd}), the determination of the center (Proposition~\ref{prop:center-uv}), and the analysis of the lower central series (Theorem~\ref{thm:structural}). Section~\ref{sec:lerf} studies subgroup separability (Theorem~\ref{thm:uv_lerf}) and the Howson property (Theorem~\ref{thm:uv_howson}), giving a complete classification. Section~\ref{sec:finite} investigates finite quotients and rigidity phenomena: we prove that the commutator subgroup is perfect for $n\ge 5$ (Proposition~\ref{prop:UV-commutator-perfect}), show that $S_n$ is the smallest non--abelian finite quotient (Theorem~\ref{thm:minimal_sn}), thereby proving Theorem~\ref{thm:B}, classify homomorphisms to symmetric groups (Theorem~\ref{thm:UV_to_Sm}), prove that $KUV_n(c)$ is characteristic (Proposition~\ref{prop:characteristic-K}), and construct finite quotients of arbitrarily large order (Corollary~\ref{cor:finite-quotients}).

\subsection*{Acknowledgments}

The author gratefully acknowledges the support of Eliane Santos, the staff of HCA, Bruno Noronha, Luciano Macedo, M\'arcio Isabella, Andreia de Oliveira Rocha, Andreia Gracielle Santana, Ednice de Souza Santos, and SMURB--UFBA (Servi\c{c}o M\'edico Universit\'ario Rubens Brasil Soares), whose support since July 2024 was essential in enabling the completion of this work. 
The author is grateful to Anthony Genevois for pointing out a mistake in a corollary about virtual cohomological dimension in an earlier version of this manuscript posted on arXiv. 
O.~O.~was partially supported by the National Council for Scientific and Technological Development (CNPq, Brazil) through a \textit{Bolsa de Produtividade} grant No.~305422/2022--7.

\section{The universal virtual braid group}\label{sec:uvnc}

We begin by introducing the main object of study: the universal virtual braid group. This group generalizes several known families of braid-like groups while retaining a rich algebraic structure. Its presentation encodes both virtual and non-virtual crossings in a uniform way.

\subsection{Definition}

\begin{defn}\label{defn:uvnc}
Let $n\ge 2$ and $c\ge 1$. The \emph{universal virtual braid group} with $n$ strands and $c$ types of crossings, denoted by $UV_n(c)$, is the group given by the presentation with
\begin{itemize}
\item \textbf{Generators:} $\sigma_{i,t}$ and $\rho_i$, where $i=1,\dots,n-1$ and $t=1,\dots,c$;
\item \textbf{Relations:}
\begin{itemize}
  \item[(PR1)] $\rho_i\rho_{i+1}\rho_i=\rho_{i+1}\rho_i\rho_{i+1}$, for $i=1,\dots,n-2$;
  \item[(PR2)] $\rho_i\rho_j=\rho_j\rho_i$, for $|i-j|\ge 2$;
  \item[(PR3)] $\rho_i^2=1$, for $i=1,\dots,n-1$;
  \item[(CR)] $\sigma_{i,t}\sigma_{j,\ell}=\sigma_{j,\ell}\sigma_{i,t}$, for all $|i-j|\ge 2$ and all $1\le t,\ell\le c$;
  \item[(MR1)] $\sigma_{i,t}\rho_j=\rho_j\sigma_{i,t}$, for $|i-j|\ge 2$ and $1\le t\le c$;
  \item[(MR2)] $\rho_i\rho_{i+1}\sigma_{i,t}=\sigma_{i+1,t}\rho_i\rho_{i+1}$, for $i=1,\dots,n-2$ and $1\le t\le c$.
\end{itemize}
\end{itemize}
By convention, for any $c$, the group $UV_1(c)$ is trivial.
\end{defn}

The choice of generators and relations in Definition~\ref{defn:uvnc} is deliberately minimal: the relations encode only the essential interactions between the \(c\) families of crossing generators and the virtual generators, while omitting any additional constraints that would specialize to a particular virtual braid-type family. This minimality ensures that all groups of interest arise as quotients of \(UV_n(c)\) by adding the relations that characterize each specific family, as shown in Proposition~\ref{prop:natural-quotients} below.

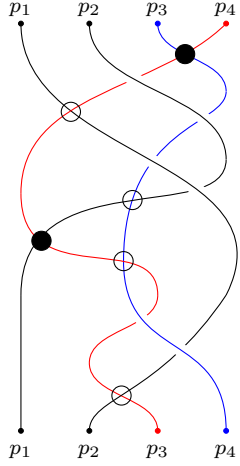
\begin{figure}[h!]
				\hfill
				\begin{tikzpicture}[scale=0.9]
					
					\draw (1,5).. controls (1,3) and (5,3) .. (4,1);
					\draw (1,1)--(1,-1);
					\draw[blue] (4,4).. controls (4,3.5) and (2.5,3.5) .. (2.5,1.5);
					\draw[red] (3,1).. controls (3,0.5) and (2,0.5).. (2,0);
					\draw[red] (2,0).. controls (2,-0.5) and (3,-0.5).. (3,-1);

					\draw (4,1).. controls (3,-0.5) and (2,-0.5) .. (2,-1);
					\draw[white,line width=5pt] (2.5,1.5).. controls (2.5,0) and (4,0.5) .. (4,-1);
					\draw[blue] (2.5,1.5).. controls (2.5,0) and (4,0.5) .. (4,-1);
					
					\draw (4,3).. controls (4,2) and (1,3).. (1,1);
					
					\draw[red] (1,2.5).. controls (1,1) and (3,2).. (3,1);
					
					\draw[white,line width=5pt] (1,5).. controls (1,3) and (5,3) .. (4,1);
					\draw (1,5).. controls (1,3) and (5,3) .. (4,1);
					\draw[white,line width=5pt] (3,5).. controls (3,4.5) and (4,4.5).. (4,4);
					\draw[blue] (3,5).. controls (3,4.5) and (4,4.5).. (4,4);
					\draw[red] (4,5).. controls (3,4) and (1,4).. (1,2.5);
					\draw[white,line width=5pt] (2,5).. controls (2,4) and (4,4).. (4,3);
					\draw (2,5).. controls (2,4) and (4,4).. (4,3);
					
					\filldraw[black] (3.4,4.55) circle (4pt);
					\filldraw[black] (1.3,1.8) circle (4pt);
					\draw (1.73,3.7) circle (4pt);
					\draw (2.63,2.4) circle (4pt);
					\draw (2.5,1.5) circle (4pt);
					\draw (2.47,-0.48) circle (4pt);
					
					
					\node at (1,5.2) {\tiny{$p_{1}$}};
					\node at (2,5.2) {\tiny{$p_{2}$}};
					\node at (3,5.2) {\tiny{$p_{3}$}};
					\node at (4,5.2) {\tiny{$p_{4}$}};
					
					\node at (1,-1.3) {\tiny{$p_{1}$}};
					\node at (2,-1.3) {\tiny{$p_{2}$}};
					\node at (3,-1.3) {\tiny{$p_{3}$}};
					\node at (4,-1.3) {\tiny{$p_{4}$}};
					
					\foreach \j in {1,2}
					{\filldraw[black] (\j,-1) circle (1pt);};
					\foreach \j in {1,2}
					{\filldraw[black] (\j,5) circle (1pt);};
			    \filldraw[blue] (3,5) circle (1pt);
			    \filldraw[blue] (4,-1) circle (1pt);
			    \filldraw[red] (4,5) circle (1pt);
			    \filldraw[red] (3,-1) circle (1pt);					
					
				\end{tikzpicture}
				\hspace*{\fill}
				\caption{A braid with 4 strands and $c=2$ types of (non-virtual) crossings. Virtual crossings are encircled by a small circle.}
			\end{figure}

There are obvious quotients of the universal virtual braid group $UV_n(c)$ like the symmetric group $S_n$, the abelianization $\mathbb{Z}^c\oplus \Z_2$, or the free abelian group $\mathbb{Z}^c$ (quotient modulo the normal closure of $v_1\in UV_n(c)$) and other virtual-braid like groups, for instance:
 \[ 
    \xymatrix@C=4.5pc{
UV_n(k) \ar@{->>}[d] \ar@{->>}[r]  & UV_n(\ell)  
  \ar@{->>}[r]
    &  UV_n(2) \ar@{->>}[d] \ar@{->>}[r] & 
  UV_n(1) \ar@{->>}[d] \ar@{->>}[rd]
   & \\
M_kVB_n &    & VSG_n \ar@{->>}[r] & VB_n & VT_n
}
 \]
where $VT_n$ is the virtual twin group \cite{BSV}, $VSG_n$ is the virtual singular braid group \cite{CPM} and $M_kVB_n$ is the multi-virtual braid group \cite{Kau2}.

\begin{prop}\label{prop:natural-quotients}
Let $n\ge 2$ and $k\ge 2$.

\begin{enumerate}
\item[(i)] The multi-virtual braid group $M_kVB_n$ is a quotient of $UV_n(k)$. 
More precisely, identifying the generators by
\[
\rho_i \mapsto \rho_i^{(0)},\qquad
\sigma_{i,t} \mapsto \rho_i^{(t)} \ (1\le t\le k-1),\qquad
\sigma_{i,k} \mapsto \sigma_i,
\]
one has
\[
M_kVB_n \cong
UV_n(k)\Big/
\Big\langle\!\Big\langle
\sigma_{i,t}\sigma_{i+1,t}\sigma_{i,t}
=
\sigma_{i+1,t}\sigma_{i,t}\sigma_{i+1,t}
\ \text{for all } t
\Big\rangle\!\Big\rangle.
\]

\item[(ii)] The virtual singular braid group $VSG_n$ is a quotient of $UV_n(2)$. 
Identifying the generators by
\[
\rho_i \mapsto \rho_i,\qquad
\sigma_{i,1} \mapsto \sigma_i,\qquad
\sigma_{i,2} \mapsto \tau_i,
\]
one obtains $VSG_n$ by imposing the defining singular braid relations between $\sigma_i$ and $\tau_i$.

\item[(iii)] The virtual twin group $VT_n$ is a quotient of $UV_n(1)$. 
More precisely,
\[
VT_n \cong
UV_n(1)\Big/
\Big\langle\!\Big\langle
\sigma_{i,1}^2
\Big\rangle\!\Big\rangle.
\]
\end{enumerate}
\end{prop}

\begin{proof}
In each case the defining relations of $UV_n(c)$ are satisfied under the indicated identification of generators. The additional relations required in the definition of $M_kVB_n$ (see \cite[Section~2]{Kau2}), $VSG_n$ (see \cite{CPM}) and $VT_n$ (see \cite{BSV}) are imposed by taking the normal closure of these relations in $UV_n(c)$.
\end{proof}

\begin{rem}
Except for the cases $n=2$ or $c=1$, the group $UV_n(c)$ is not a virtual Artin group in the sense of \cite{BPT}.
\end{rem}

\subsection{Natural quotients and maps to $S_n$}\label{subsec:quotients-maps}

Let $S_n$ be the symmetric group with generators $s_i=(i\ i+1)$, $i=1,\dots,n-1$, and the usual Coxeter relations. 
There are two natural surjective homomorphisms from $UV_n(c)$ onto $S_n$:
\[
\pi_n^P\colon UV_n(c)\to S_n,
\qquad
\pi_n^K\colon UV_n(c)\to S_n,
\]
defined by
\[
\pi_n^P(\sigma_{i,t})=s_i,\ \ \pi_n^P(\rho_i)=s_i,
\qquad
\pi_n^K(\sigma_{i,t})=1,\ \ \pi_n^K(\rho_i)=s_i,
\]
for all $i=1,\dots,n-1$ and $t=1,\dots,c$.

\begin{defn}\label{defn:puvkuv}
Let
\[
PUV_n(c)=\ker{\pi_n^P}
\quad\text{and}\quad
KUV_n(c)=\ker{\pi_n^K}.
\]
We call $PUV_n(c)$ the \emph{pure universal virtual braid group}, and $KUV_n(c)$ the \emph{kernel universal virtual braid group}.
\end{defn}

The virtual generators $\rho_i$ satisfy the defining relations of the symmetric group. This observation leads to two natural surjective homomorphisms onto $S_n$, each of which splits. As a consequence, $UV_n(c)$ admits semidirect product decompositions that reveal its underlying structure.

\begin{prop}\label{prop:semidirect}
For all $n\ge 2$ and $c\ge 1$, the group $UV_n(c)$ admits decompositions as semidirect products
\[
UV_n(c)=PUV_n(c)\rtimes S_n
\quad\text{and}\quad
UV_n(c)=KUV_n(c)\rtimes S_n.
\]
\end{prop}

\begin{proof}
Define a map $\iota\colon S_n\to UV_n(c)$ by $\iota(s_i)=\rho_i$. Relations (PR1)--(PR3) show that $\iota$ is a well-defined injective homomorphism.
Moreover, $\pi_n^P\circ \iota=\pi_n^K\circ \iota=\mathrm{id}_{S_n}$. Therefore each short exact sequence
\[
1\to \ker{\pi}\to UV_n(c)\stackrel{\pi}{\to} S_n\to 1
\]
splits, for $\pi=\pi_n^P,\pi_n^K$, and the result follows.
\end{proof}

\subsection{Inheritance of the semidirect structure}\label{subsec:inheritance}

The semidirect product decompositions of $UV_n(c)$ are inherited by quotients that preserve the virtual generators.

\begin{prop}\label{prop:inheritance-semidirect}
Let $Q=UV_n(c)/N$ be a quotient such that the canonical image of the subgroup
$\langle\rho_1,\dots,\rho_{n-1}\rangle$ in $Q$ is isomorphic to $S_n$
(equivalently, $N\cap\langle\rho_1,\dots,\rho_{n-1}\rangle=\{1\}$).
Let $\pi\colon UV_n(c)\twoheadrightarrow Q$ be the quotient map, and denote
$\bar\rho_i=\pi(\rho_i)$.

Then the homomorphism $\bar\pi_n^K\colon Q\to S_n$ induced by $\pi_n^K$
(i.e., $\bar\pi_n^K(\bar\rho_i)=s_i$ and $\bar\pi_n^K(\pi(\sigma_{i,t}))=1$)
is surjective and splits via $\bar\iota\colon S_n\to Q$, $\bar\iota(s_i)=\bar\rho_i$.
Consequently,
\[
Q\cong \operatorname{Ker}(\bar\pi_n^K)\rtimes S_n,
\]
where $\operatorname{Ker}(\bar\pi_n^K)=\pi(KUV_n(c))$.
\end{prop}

\begin{proof}
Because $N$ does not intersect the virtual subgroup, the restriction $\pi|_{\langle\rho_i\rangle}$ is injective; hence the elements $\bar\rho_i$ satisfy the Coxeter relations (PR1)--(PR3) and generate a subgroup isomorphic to $S_n$. The map $\bar\iota\colon s_i\mapsto\bar\rho_i$ is therefore an injective homomorphism. By definition, $\bar\pi_n^K\circ\bar\iota=\operatorname{id}_{S_n}$, so $\bar\iota$ is a section for $\bar\pi_n^K$. The splitting of the resulting short exact sequence yields the claimed semidirect product decomposition.
\end{proof}

This structural uniformity is the reason why many algebraic and geometric properties deduced from the semidirect decomposition of $UV_n(c)$ are automatically transferred to all quotients that retain an intact copy of $S_n$.

\subsection{A finite index right--angled Artin subgroup}\label{subsec:raag}

A key insight in understanding the structure of $UV_n(c)$ is the existence of a right-angled Artin subgroup of finite index. In this subsection we construct such a subgroup explicitly by considering the kernel of the homomorphism that kills all non-virtual crossings. Throughout, we keep the notation of \redefn{uvnc} and \redefn{puvkuv}. Recall that the homomorphism $\pi_n^K\colon UV_n(c)\to S_n$ satisfies $\pi_n^K(\rho_i)=s_i$ and $\pi_n^K(\sigma_{i,t})=1$, hence $KUV_n(c)=\ker{\pi_n^K}$.

\begin{defn}\label{defn:anc}
Let $U_n(c)$ denote the group with generators $\sigma_{i,t}$,
$i=1,\dots,n-1$, $t=1,\dots,c$, and defining relations
\[
\sigma_{i,t}\sigma_{j,\ell}=\sigma_{j,\ell}\sigma_{i,t}
\qquad\text{for}\qquad |i-j|\ge 2,\ \ 1\le t,\ell\le c.
\]
Equivalently, $U_n(c)$ is the right--angled Artin group associated to the graph whose vertex set is $\{\sigma_{i,t}\}$ and where two vertices are adjacent if and only if the corresponding generators commute.
\end{defn}

\begin{rem}\label{rem:anc_in_uv}
The defining relations of $U_n(c)$ are precisely the relations (CR) in \redefn{uvnc}. Therefore the assignment $\sigma_{i,t}\mapsto\sigma_{i,t}$ induces an injective homomorphism $U_n(c)\hookrightarrow UV_n(c)$, and we identify $U_n(c)$ with its image. Moreover, by definition of $\pi_n^K$, one has $U_n(c)\le KUV_n(c)$.
\end{rem}

\subsubsection*{Conjugates of the non-virtual generators}

For $1\le i\neq j\le n$ and $1\le t\le c$, we define elements $\delta_{i,j,t}\in KUV_n(c)$ as follows. If $j=i+1$, set
\[
\delta_{i,i+1,t}=\sigma_{i,t}.
\]
If $j> i+1$, set
\[
\delta_{i,j,t}
=
\rho_{j-1}\rho_{j-2}\cdots \rho_{i+1}\,
\sigma_{i,t}\,
\rho_{i+1}\cdots \rho_{j-2}\rho_{j-1},
\]
and
\[
\delta_{j,i,t}
=
\rho_{j-1}\rho_{j-2}\cdots \rho_{i+1}\,
\rho_i\,\sigma_{i,t}\,\rho_i\,
\rho_{i+1}\cdots \rho_{j-2}\rho_{j-1}.
\]
By construction, $\delta_{i,j,t}\in KUV_n(c)$ for all admissible indices.

\begin{lem}\label{lem:actionkuvnc}
Let $w\in S_n$ and let $1\le i\neq j\le n$ and $1\le t\le c$. Under the conjugation action of $\iota(S_n)\le UV_n(c)$ (cf. Proposition~\ref{prop:semidirect}), one has
\[
\iota(w)\,\delta_{i,j,t}\,\iota(w)^{-1}=\delta_{w(i),\,w(j),\,t}.
\]
\end{lem}

\begin{proof}
It is enough to check the claim for $w=s_k$. Using (PR1)--(PR3), the element $\rho_k$ acts on the indices by the adjacent transposition $(k\ k+1)$, and the relations (MR1)--(MR2) allow one to slide $\rho_k$ past $\sigma_{i,t}$, producing exactly the reindexing prescribed by $(k\ k+1)$. The general case follows by writing $w$ as a word in the $s_k$'s.
\end{proof}

We recall that a \emph{right-angled Artin group} (RAAG) is a group defined by a presentation where the only relations are commutations between certain generators.  More concretely, given a finite graph $\Gamma$ with vertex set $V$, the RAAG associated to $\Gamma$ is 
\[
A(\Gamma) = \langle V \mid [v,w] = 1 \text{ for all } \{v,w\} \in E(\Gamma) \rangle.
\]
When the graph has no edges, $A(\Gamma)$ is a free group; when the graph is complete, $A(\Gamma)$ is a free abelian group.  Right-angled Artin groups are central objects in geometric group theory due to their rich subgroup structure and their role in the study of cubulated groups; see, for instance, \cite{Davis2008}.

We now turn to a detailed study of $KUV_n(c)$, the kernel of $\pi_n^{K}$.  By examining the conjugation action of the symmetric group on the non-virtual generators, we obtain a generating set whose commutation relations are controlled by disjointness of indices. This leads to the following presentation.

\begin{thm}\label{thm:preskuvnc}
Let $n\ge 2$ and $c\ge 1$. The group $KUV_n(c)$ is a right--angled Artin group. More precisely, it admits a presentation with generating set
\[
\{\delta_{i,j,t}\mid 1\le i\neq j\le n,\ 1\le t\le c\}
\]
and defining relations
\[
\delta_{i,j,t}\,\delta_{k,\ell,m}=\delta_{k,\ell,m}\,\delta_{i,j,t},
\]
whenever the indices $i,j,k,\ell$ are pairwise distinct and
$1\le t,m\le c$.
\end{thm}

\begin{proof}
By \rerem{anc_in_uv}, the subgroup $U_n(c)$ lies in $KUV_n(c)$, and $KUV_n(c)$ is the normal closure of $U_n(c)$ in $UV_n(c)$ (since $\pi_n^K$ kills all $\sigma_{i,t}$ and maps $\rho_i$ onto $S_n$).
Hence $KUV_n(c)$ is generated by all conjugates of the $\sigma_{i,t}$ by elements of $\iota(S_n)$, which are precisely the elements $\delta_{i,j,t}$ defined above. This proves the generating set.

Now assume that $i,j,k,\ell$ are pairwise distinct. Using \relem{actionkuvnc}, we may conjugate so as to reduce to the case where $\delta_{i,j,t}=\sigma_{p,t}$ and $\delta_{k,\ell,m}=\sigma_{q,m}$ with $|p-q|\ge 2$. Then the commutation relation follows from (CR). 
Conversely, if two generators $\delta_{i,j,t}$ and $\delta_{k,\ell,m}$ do not satisfy the pairwise distinct condition, then the sets $\{i,j\}$ and $\{k,\ell\}$ intersect. In this case, no relation in the presentation of $UV_n(c)$ forces these generators to commute; indeed, such a commutation would contradict the action of $S_n$ (Lemma~\ref{lem:actionkuvnc}) or the mixed relations (MR1)-(MR2). Hence no commutation relation is imposed. 
Therefore, the only relations among the generators $\delta_{i,j,t}$ are the commutation relations $\delta_{i,j,t}\delta_{k,\ell,m}=\delta_{k,\ell,m}\delta_{i,j,t}$ precisely when $\{i,j\}$ and $\{k,\ell\}$ are disjoint. This is exactly the defining presentation of a right-angled Artin group. 
\end{proof}

The fact that \(KUV_n(c)\) is a right-angled Artin group is a direct consequence of the minimality  of the relations defining \(UV_n(c)\): the only constraints among the crossing generators  \(\sigma_{i,t}\) are commutation relations between generators that involve disjoint pairs of strands,  and the semidirect product structure ensures that the subgroup generated by the conjugates of the  \(\sigma_{i,t}\) under the virtual generators is free abelian in each block. This RAAG structure is precisely what allows many algebraic and geometric properties to be established in a unified manner for \(UV_n(c)\) and subsequently transferred to its quotients that preserve the virtual generators, as we will see in Sections~\ref{sec:structural} and~\ref{sec:finite}.

\begin{cor}\label{cor:kuv_small}
Let $c\ge 1$.
\begin{enumerate}
\item $KUV_2(c)$ is a free group of rank $2c$.
\item $KUV_3(c)$ is a free group of rank $6c$.
\item Let $n \ge 4$ and $c \ge 1$. 
Then the right-angled Artin group $KUV_n(c)$ contains a subgroup isomorphic to $F_2 \times F_2$, where $F_2$ denotes the free group of rank $2$.
\end{enumerate}
\end{cor}

\begin{proof}

For $n=2$ (resp. $n=3$), the defining graph of the RAAG in \rethm{preskuvnc} has no edges, hence the group is free of rank $2c$ (resp. $6c$). 

Now we move to the case when $n \ge 4$ and $c \ge 1$.  By Theorem \ref{thm:preskuvnc}, the group $KUV_n(c)$ admits a presentation with generating set 
\[
\{\delta_{i,j,t} \mid 1 \le i \neq j \le n,\; 1 \le t \le c\}
\]
and defining relations
\[
\delta_{i,j,t}\,\delta_{k,\ell,m} = \delta_{k,\ell,m}\,\delta_{i,j,t},
\]
whenever $i,j,k,\ell$ are pairwise distinct and $1\le t,m\le c$.

Since $n \ge 4$, we can choose four distinct indices $a,b,p,q \in \{1,\dots,n\}$. 
Fix a colour $t_0 \in \{1,\dots,c\}$ and define
\[
X_1 = \delta_{a,b,t_0}, \qquad X_2 = \delta_{b,a,t_0},
\]
\[
Y_1 = \delta_{p,q,t_0}, \qquad Y_2 = \delta_{q,p,t_0}.
\]

The set $\{a,b,p,q\}$ has four distinct elements, hence the defining relations imply that  every $X_i$ commutes with every $Y_j$:
\[
[X_i, Y_j] = 1 \quad \text{for } i,j \in \{1,2\}.
\]

There are no defining relations involving only $X_1$ and $X_2$ (they share the indices $a,b$),  so $\langle X_1, X_2 \rangle$ is a free group of rank~$2$. 
Similarly, $\langle Y_1, Y_2 \rangle$ is a free group of rank~$2$.  Because all elements of $\{X_1,X_2\}$ commute with all elements of $\{Y_1,Y_2\}$,  the subgroup generated by $X_1,X_2,Y_1,Y_2$ is isomorphic to the direct product 
\[
\langle X_1, X_2 \rangle \times \langle Y_1, Y_2 \rangle \cong F_2 \times F_2.
\]
Thus $KUV_n(c)$ contains $F_2 \times F_2$ as a subgroup.
\end{proof}

\begin{rem}\label{rem:index}
Since $\pi_n^K$ is surjective and admits the section $\iota\colon S_n\to UV_n(c)$, it follows that
\[
[UV_n(c)\colon KUV_n(c)]=|S_n|=n!,
\]
and $UV_n(c)=KUV_n(c)\rtimes S_n$ (Proposition~\ref{prop:semidirect}).
\end{rem}


\section{Structural and residual properties}\label{sec:structural}

In this section we investigate algebraic, residual and algorithmic properties of the universal virtual braid group $UV_n(c)$.
The key structural feature is that $UV_n(c)$ is virtually a right--angled Artin group. This single fact implies a number of strong consequences.
We also analyse further properties that rely directly on the defining presentation.

\subsection{Properties of universal virtual braid groups}

We begin with elementary structural information.

\begin{prop}\label{prop:basic}
Let $n\ge 2$ and $c\ge 1$.
\begin{enumerate}
\item The group $UV_2(c)$ is isomorphic to the free product
\[
UV_2(c)\cong F_c\ast \Z_2,
\]
where $F_c$ is the free group of rank $c$.

\item For every $n\ge 2$, the abelianization of $UV_n(c)$ is
\[
UV_n(c)^{ab}\cong \Z^c\oplus \Z_2 .
\]
\end{enumerate}
\end{prop}

\begin{proof}
Let $n\ge 2$ and $c\ge 1$. 
\begin{enumerate}
    \item For $n=2$, the presentation has generators $\sigma_{1,t}$ ($t=1,\dots,c$) and $\rho_1$, with the only relation $\rho_1^2=1$.
Hence
\[
UV_2(c)\cong \ang{\sigma_{1,1},\dots,\sigma_{1,c}}
\ast
\ang{\rho_1},
\]
which is $F_c\ast \Z_2$.

\item In the abelianization all generators commute. Relation (PR3) implies that the class of each $\rho_i$ has order two. Relations (MR2) identify all $\sigma_{i,t}$ with fixed $t$. 
Thus there are $c$ independent infinite cyclic generators coming from the non-virtual crossings and one element of order two coming from the virtual part.
\end{enumerate}
\end{proof}

\begin{prop}\label{prop:UV-commutator-perfect}
Let $n\ge 5$ and $c\ge 1$. Let $G=UV_n(c)$ and let $N=\langle\!\langle \rho_1\rho_3\rangle\!\rangle$ be the normal closure of $\rho_1\rho_3$ in $G$. 
Then $N=G'$, the commutator subgroup of $G$, and $G'$ is perfect.
\end{prop}

\begin{proof}
Let $G=UV_n(c)$ and $N=\langle\!\langle \rho_1\rho_3\rangle\!\rangle$. 
Set $Q=G/N$.

\smallskip
\noindent\emph{Step 1: $Q\cong G^{ab}$.}
In the subgroup generated by $\rho_1,\dots,\rho_{n-1}$ we have the standard presentation of $S_n$.
Since $\rho_1$ and $\rho_3$ commute, the relation $\rho_1\rho_3=1$ in $Q$ yields $\rho_1=\rho_3$ in $Q$.
By conjugation in $S_n$ it follows that all $\rho_i$ have the same image in $Q$; denote it by $r$.
By (PR3), we have $r^2=1$.

Using (MR2) in $Q$, for each $t$ and each $i=1,\dots,n-2$ we get
\[
\rho_i\rho_{i+1}\sigma_{i,t}=\sigma_{i+1,t}\rho_i\rho_{i+1}.
\]
Since $\rho_i=\rho_{i+1}=r$ in $Q$, we have $\rho_i\rho_{i+1}=r^2=1$, hence $\sigma_{i,t}=\sigma_{i+1,t}$ in $Q$.
Thus, for each $t$ the elements $\sigma_{i,t}$ have a common image; denote it by $s_t$.

Applying (CR) with $(i,j)=(1,3)$ (which is allowed since $|1-3|=2$) yields, for all $t,\ell$,
\[
\sigma_{1,t}\sigma_{3,\ell} = \sigma_{3,\ell}\sigma_{1,t}.
\]
Passing to $Q$ gives $s_t s_\ell=s_\ell s_t$ for all $t,\ell$, so $\langle s_1,\dots,s_c\rangle\cong \mathbb Z^c$.
Moreover, applying (MR1) with $(i,j)=(3,1)$ yields $\sigma_{3,t}\rho_1=\rho_1\sigma_{3,t}$, hence $s_t r=r s_t$.
Therefore $Q\cong \mathbb Z^c\oplus \mathbb Z_2$.

By Proposition~\ref{prop:basic} (abelianization of $UV_n(c)$), we have
$G^{ab}\cong \mathbb Z^c\oplus \mathbb Z_2$, hence $Q\cong G^{ab}$.
Consequently, $G' \le N$.

\smallskip
\noindent\emph{Step 2: $N\le [G',G']$.}
Following an observation of Lin (see \cite[Remark~1.10]{Lin}), in the Coxeter-type subgroup generated by the $\rho_i$ we have the identity
\[
\rho_3\rho_1=(\rho_1\rho_2)^{-1}\,[\rho_3\rho_1,\rho_1\rho_2]\,(\rho_1\rho_2),
\]
where $[a,b]=a^{-1}b^{-1}ab$.
In $G^{ab}$ all $\rho_i$ have the same class of order $2$, hence both $\rho_1\rho_2$ and $\rho_3\rho_1$ map to the identity in $G^{ab}$, and therefore belong to $G'$.
It follows that $\rho_3\rho_1\in [G',G']$. Since $[G',G']$ is characteristic in $G$, the normal closure $N=\langle\!\langle \rho_1\rho_3\rangle\!\rangle$ is contained in $[G',G']$, i.e. $N\le [G',G']$.

\smallskip
Combining the inclusions $G'\le N\le [G',G']\le G'$, we obtain $N=G'=[G',G']$.
Thus $G'$ is perfect.
\end{proof}

The structural properties of $UV_n(c)$ are largely governed by the fact that it is virtually a right--angled Artin group. 
Recall that a group satisfies the \emph{Tits alternative} if every finitely generated subgroup either contains a non-abelian free group or is virtually solvable.
Classical braid groups and right-angled Artin groups are well-known examples of groups satisfying this property.

\begin{thm}\label{thm:virtually-raag}
Let $n\ge 2$ and $c\ge 1$. The group $UV_n(c)$ is virtually a right--angled Artin group. More precisely, the subgroup $KUV_n(c)$ is a right--angled Artin group of finite index in $UV_n(c)$.

Consequently: 
\begin{enumerate}
\item $UV_n(c)$ is linear and residually finite;

\item the word and conjugacy problems in $UV_n(c)$ are solvable;

\item $UV_n(c)$ satisfies the Tits alternative;

\item $UV_n(c)$ has exponential growth.
\end{enumerate}
\end{thm}

\begin{proof}
By Theorem~\ref{thm:preskuvnc}, the subgroup $KUV_n(c)$ is a right--angled Artin group.
By Remark~\ref{rem:index}, it has finite index $n!$ in $UV_n(c)$.

Right--angled Artin groups are linear and residually finite, have solvable word and conjugacy problems, and satisfy the Tits alternative. 
These properties are all preserved under taking finite extensions (see e.g. \cite{Bogopolski2008} for detailed treatments).
Hence they hold for $UV_n(c)$.

Finally, the group $KUV_n(c)$ is non--abelian for all $n\ge 2$ and $c\ge 1$ (with the sole exception of the degenerate case $n=2$, $c=0$, which does not occur), and therefore has exponential growth.
Exponential growth is preserved under finite extensions (see e.g. \cite{Bogopolski2008}). 
\end{proof}

We now collect several structural properties of $UV_n(c)$ that follow either from its virtual RAAG structure or directly from its presentation. For a group $G$, the \emph{lower central series} is defined inductively by $\Gamma_1(G)=G$ and $\Gamma_{k+1}(G)=[G,\Gamma_k(G)]$, the subgroup generated by commutators $[x,y]=x^{-1}y^{-1}xy$ with $x\in G$, $y\in\Gamma_k(G)$.

\begin{thm}\label{thm:structural}
Let $n\ge 2$ and $c\ge 1$.
\begin{enumerate}
\item The group $UV_n(c)$ is Hopfian.
\item \begin{enumerate}
      \item If $n=2,3$, the lower central series of $UV_n(c)$ does not stabilize.
      \item If $n\ge 4$, then $\Gamma_2(UV_n(c)) = \Gamma_3(UV_n(c))$, and in particular $UV_n(c)$ is not residually nilpotent.
      \end{enumerate}
\end{enumerate}
\end{thm}

\begin{proof}

Let $n\ge 2$ and $c\ge 1$.
\begin{enumerate}
    \item  Since $UV_n(c)$ is finitely generated and residually finite (Theorem~\ref{thm:virtually-raag}), it is Hopfian.

\item The argument follows the strategy used for virtual braid groups (see \cite{BB}).

For $n\ge 4$, relation (MR2) gives
\[
\rho_i\rho_{i+1}\sigma_{i,t}
=
\sigma_{i+1,t}\rho_i\rho_{i+1}.
\]
Rewriting this as a commutator,
\[
[\sigma_{i,t}^{-1}\sigma_{i+1,t},\,
\rho_i\rho_{i+1}]
=
\sigma_{i,t}^{-1}\sigma_{i+1,t}.
\]
Thus $\sigma_{i,t}^{-1}\sigma_{i+1,t}$ belongs to $\Gamma_2(UV_n(c))$ and, being a commutator, also to $\Gamma_3(UV_n(c))$.
It follows that
$\sigma_{i,t}\equiv \sigma_{i+1,t}
\pmod{\Gamma_3(UV_n(c))}$. 
Iterating, all non-virtual generators become equal modulo $\Gamma_3(UV_n(c))$. 
Since $\Gamma_2(UV_n(c))$ is generated by commutators of the form $\sigma_{i,t}^{-1}\sigma_{i+1,t}$, which lie in $\Gamma_3(UV_n(c))$, we have $\Gamma_2(UV_n(c)) \subseteq \Gamma_3(UV_n(c))$. 
The reverse inclusion $\Gamma_3(UV_n(c)) \subseteq \Gamma_2(UV_n(c))$ always holds, hence $\Gamma_2(UV_n(c)) = \Gamma_3(UV_n(c))$.

For $n=2,3$, the subgroup generated by the non-virtual generators is free (Corollary~\ref{cor:kuv_small}), and its lower central series does not stabilise.
This non--stabilisation persists in $UV_n(c)$.
\end{enumerate}
\end{proof}

\begin{rem}
The fact that $UV_n(c)$ is virtually a right--angled Artin group places it in the framework of virtually special groups in the sense of Haglund--Wise \cite{HW2008}. 
In particular, $UV_n(c)$ is virtually compact special and therefore admits a proper and cocompact action on a CAT(0) cube complex (see \cite{Wise2012}). This structural viewpoint provides additional information on separability properties, subgroup structure and residual behaviour, and situates the universal virtual braid group within the modern theory of cubulated groups.
\end{rem}

\subsection{Virtual cohomological dimension}
\label{subsec:vcd}

Recall that for a group $G$ that is virtually torsion-free, the \emph{virtual cohomological dimension} of $G$, denoted $\operatorname{vcd}(G)$, is defined as the cohomological dimension of any torsion-free subgroup of finite index; this is well-defined independently of the chosen subgroup (see, for example, \cite{Bieri1982}). 
Since $KUV_n(c)$ is torsion-free and has finite index $n!$ in $UV_n(c)$ (Remark~\ref{rem:index}), we have
\[
\operatorname{vcd}(UV_n(c)) = \operatorname{cd}(KUV_n(c)).
\]
We now compute this dimension.

\begin{thm}\label{thm:vcd}
Let $n\ge 2$ and $c\ge 1$. Then
\[
\operatorname{vcd}\bigl(UV_n(c)\bigr)=\Bigl\lfloor\frac{n}{2}\Bigr\rfloor.
\]
\end{thm}

\begin{proof}
By Theorem~\ref{thm:preskuvnc}, the group $KUV_n(c)$ is the RAAG associated to the graph $\Gamma_{n,c}$ whose vertices are the generators $\delta_{i,j,t}$ ($1\le i\neq j\le n$, $1\le t\le c$), and where two vertices are adjacent if and only if the unordered pairs $\{i,j\}$ and $\{k,\ell\}$ are disjoint.  
For a RAAG $A(\Gamma)$, one has $\operatorname{cd}(A(\Gamma))=\omega(\Gamma)$, the size of the largest clique in $\Gamma$ (see \cite[Corollary~5.5]{Davis2008}). 
Recall that a \emph{clique} in a graph is a set of vertices pairwise joined by an edge.

A clique in $\Gamma_{n,c}$ corresponds to a collection of vertices $\delta_{i_1,j_1,t_1},\dots,\delta_{i_r,j_r,t_r}$ such that the unordered pairs $\{i_s,j_s\}$ are pairwise disjoint. Hence $2r\le n$, so $r\le\lfloor n/2\rfloor$.
Conversely, choosing $\lfloor n/2\rfloor$ disjoint unordered pairs in $\{1,\dots,n\}$ yields a clique of size $\lfloor n/2\rfloor$ in $\Gamma_{n,c}$ (by selecting any orientation and any type for each pair). Thus $\omega(\Gamma_{n,c})=\lfloor n/2\rfloor$, and the claim follows.
\end{proof}

\begin{rem}\label{rem:vcd-examples}
For the specific families that motivated the definition of $UV_n(c)$, the virtual cohomological dimension is known in some cases.

\begin{itemize}
\item \textbf{Virtual braid groups $VB_n$.} 
Godelle and Paris proved that $\operatorname{vcd}(VB_n)=n-1$ \cite[Corollary~6.3]{GP2012}.

\item \textbf{Virtual singular braid groups $VSG_n$.} 
The virtual cohomological dimension of $VSG_n$ has not been computed in the literature. 
However, by \cite[Theorem~16]{Ocampo}, there exists a split epimorphism $VSG_n \to VB_n$ (sending $\tau_i \mapsto 1$) with section $VB_n \hookrightarrow VSG_n$ given by the natural inclusion of generators. 
Hence $VB_n$ is isomorphic to a subgroup of $VSG_n$, and since $\operatorname{vcd}$ is monotone under taking subgroups (see \cite[Proposition~2.4(a)]{Brown1982}), we obtain $\operatorname{vcd}(VSG_n) \ge \operatorname{vcd}(VB_n) = n-1$.

\item \textbf{Virtual twin groups $VT_n$.} 
Naik, Nanda, and Singh proved that the pure virtual twin group $PVT_n$ is a right-angled Artin group whose defining graph has vertices indexed by unordered pairs $\{i,j\}$ with $1 \le i < j \le n$, where two vertices commute if and only if the corresponding pairs are disjoint \cite[Corollary~3.4]{NNS2023}. 
For a right-angled Artin group, the cohomological dimension equals the size of the largest clique in its defining graph \cite[Corollary~5.5]{Davis2008}. 
The largest clique in this graph consists of $\lfloor n/2\rfloor$ pairwise disjoint pairs, hence $\operatorname{cd}(PVT_n) = \lfloor n/2\rfloor$. 
Since $VT_n$ is a finite extension of $PVT_n$ (see \cite[Proposition~3.1]{NNS2023}), it follows that $\operatorname{vcd}(VT_n) = \lfloor n/2\rfloor$.

\item \textbf{Multi-virtual braid groups $M_kVB_n$.} 
As far as we know, the virtual cohomological dimension of $M_kVB_n$ has not been studied in the literature. No nontrivial bounds are currently known.
\end{itemize}

Thus, among the natural quotients of $UV_n(c)$ that preserve the virtual generators, the virtual cohomological dimension varies significantly: it can be as large as $n-1$ (for $VB_n$ and $VSG_n$) or as small as $\lfloor n/2\rfloor$ (for $VT_n$). The general question of determining $\operatorname{vcd}(Q)$ for an arbitrary quotient $Q = UV_n(c)/N$ (with $N \cap \langle \rho_1,\dots,\rho_{n-1}\rangle = \{1\}$) remains open.
\end{rem}

\subsection{Center of $UV_n(c)$}

We now determine the center of the universal virtual braid group. 
Recall that $KUV_n(c)$ is a right--angled Artin group whose defining graph $\Gamma_{n,c}$ has vertices $\delta_{i,j,t}$ with $1\le i\neq j\le n$, $1\le t\le c$, and where two vertices are adjacent if and only if the corresponding unordered pairs $\{i,j\}$ and $\{k,\ell\}$ are disjoint.

For $n\ge 4$, the graph $\Gamma_{n,c}$ is connected and not complete, and a connected right--angled Artin group with non-complete graph has trivial center (see \cite[Corollary~5.5]{Davis2008}).
For $n=2,3$, $KUV_n(c)$ is free of rank $2c$ and $6c$, respectively, and therefore has trivial center as well (a non-abelian free group has trivial center; the abelian case $c=0$ does not occur).

\begin{prop}\label{prop:center-uv}
Let $n\ge 2$ and $c\ge 1$. Then $Z(UV_n(c))=1$.
\end{prop}

\begin{proof}
We first note that $Z(UV_n(c))\cap KUV_n(c)\subseteq Z(KUV_n(c))$.
As argued above, 
$$
Z(KUV_n(c))=1 \textrm{ for all } n\ge 2,\, c\ge 1
$$ 
(for $n=2,3$ because $KUV_n(c)$ is free of rank $\ge 2$; for $n\ge 4$ because its defining graph is connected and non-complete).
Thus $Z(UV_n(c))\cap KUV_n(c)=1$.

Since $KUV_n(c)$ has finite index $n!$ in $UV_n(c)$ (Remark~\ref{rem:index}), any element $z\in Z(UV_n(c))$ has finite order when projected to $UV_n(c)/KUV_n(c)\cong S_n$, and hence $z$ lies in a finite subgroup of $UV_n(c)$ (the preimage of a finite subgroup of $S_n$ under the projection).
On the other hand, the subgroup $\iota(S_n)$ generated by the $\rho_i$ is a complement of $KUV_n(c)$ in $UV_n(c)$, and its non-trivial elements do not centralize $KUV_n(c)$ for $n\ge 3$ (for instance, $\rho_1$ does not commute with $\delta_{1,2,1}=\sigma_{1,1}$ by (MR2)). Hence no non-trivial element of $\iota(S_n)$ can lie in the center. 
Consequently, $Z(UV_n(c))=1$ for $n\ge 3$.

For $n=2$, $UV_2(c)\cong F_c*\mathbb{Z}_2$ (Proposition~\ref{prop:basic}). 
A free product of two non-trivial groups has trivial center unless both factors are of order $2$ and the free product is the infinite dihedral group, but here $F_c$ is free of rank $c\ge 1$, so the center is trivial.
\end{proof}

\section{Subgroup separability}\label{sec:lerf}

In this section we investigate subgroup separability (also known as the LERF property) for the universal virtual braid group $UV_n(c)$ and its distinguished subgroups.  
A group $G$ is said to be \emph{LERF} (locally extended residually finite) if every finitely generated subgroup of $G$ is closed in the profinite topology; equivalently, every finitely generated subgroup is an intersection of finite-index subgroups. 
(Recall that the \emph{profinite topology} on a group is the topology whose basis of open sets consists of cosets of finite-index subgroups.) This property is a natural strengthening of residual finiteness and has important consequences for the solubility of the generalized word problem (see \cite{ALO} for the case of braid groups).

The presence of a right-angled Artin subgroup of finite index in $UV_n(c)$ (Theorem~\ref{thm:preskuvnc}) allows us to apply known results on separability for RAAGs and free groups, leading to a complete classification of the LERF property in terms of the parameters $n$ and $c$. 
We begin by recalling a classical fact that will be used repeatedly throughout this section.

The following theorem, due to Scott \cite{Scott78}, describes how the LERF property behaves under passage to finite-index subgroups and extensions. 
It will be our main tool for transferring information between $UV_n(c)$ and its finite-index subgroups $KUV_n(c)$ and $PUV_n(c)$.

\begin{thm}[Scott]\label{thm:lerf_scott}
Let $H$ be a subgroup of a group $G$.
\begin{enumerate}
    \item If $H$ is not LERF, then $G$ is not LERF.
    \item If $[G\colon H]<\infty$, then $G$ is LERF if and only if $H$ is LERF.
\end{enumerate}
\end{thm}

We also recall that free groups are LERF, while the group $F_2\times F_2$ is not LERF \cite[Lemma~8]{ALO}. 
By Theorem~\ref{thm:preskuvnc}, the group $KUV_n(c)$ is a right--angled Artin group (RAAG). 
Our strategy is to combine the LERF property of free groups with the obstruction provided by $F_2 \times F_2$.

\begin{prop}\label{prop:kuv_lerf}
Let $n\ge 2$ and $c\ge 1$.
The group $KUV_n(c)$ is subgroup separable if and only if $n=2$ or $n=3$.
\end{prop}

\begin{proof}
If $n=2$ or $n=3$, then by items (1) and (2) of Corollary~\ref{cor:kuv_small} the group $KUV_n(c)$ is free, hence LERF.

Now assume $n\ge 4$. 
By item (3) of Corollary~\ref{cor:kuv_small}, for any $c\ge 1$ the group $KUV_n(c)$ contains a subgroup isomorphic to $F_2 \times F_2$. 
Since $F_2 \times F_2$ is not LERF \cite[Lemma~8]{ALO}, Theorem~\ref{thm:lerf_scott} implies that $KUV_n(c)$ cannot be LERF.
\end{proof}

We now transfer the previous result to the full universal virtual braid group and its pure subgroup.

\begin{thm}\label{thm:uv_lerf}
Let $n\ge 2$ and $c\ge 1$.
\begin{enumerate}
    \item The universal virtual braid group $UV_n(c)$ is subgroup separable if and only if $n=2$ or $n=3$.
    \item The pure universal virtual braid group $PUV_n(c)$ is subgroup separable if and only if $n=2$ or $n=3$.
\end{enumerate}
\end{thm}

\begin{proof}
By Proposition~\ref{prop:semidirect}, the subgroup $KUV_n(c)$ has finite index in $UV_n(c)$. 
Since $PUV_n(c)$ also has finite index in $UV_n(c)$, Scott's theorem (Theorem~\ref{thm:lerf_scott}) implies that $UV_n(c)$ and $PUV_n(c)$ are LERF if and only if $KUV_n(c)$ is LERF. 
The conclusion now follows directly from Proposition~\ref{prop:kuv_lerf}.
\end{proof}

\begin{rem}
The obstruction to subgroup separability for $n\ge 4$ comes from the fact that $KUV_n(c)$ always contains $F_2\times F_2$, which is not LERF. 
Consequently, the only LERF universal virtual braid groups are those with $n\le 3$; note that for $n=2,3$ the group $UV_n(c)$ is virtually free (since $KUV_n(c)$ is free of finite index), and virtually free groups are LERF.
\end{rem}

\subsection{The Howson property}\label{subsec:howson}

A group $G$ is said to have the \emph{Howson property} (or to be \emph{Howson}) if the intersection of any two finitely generated subgroups of $G$ is again finitely generated. 
This property was first established for free groups by Howson~\cite{Howson54} and has since become a classical finiteness condition in combinatorial and geometric group theory.

We begin with a standard lemma that allows us to transfer the Howson property between a group and its finite-index subgroups.

\begin{lem}\label{lem:howson_transfer}
Let $G$ be a group and let $H$ be a subgroup of finite index in $G$.
\begin{enumerate}
    \item If $H$ is Howson, then $G$ is Howson.
    \item If $G$ is Howson, then $H$ is Howson.
\end{enumerate}
\end{lem}

\begin{proof}
\begin{enumerate}
    \item Let $A,B\le G$ be finitely generated subgroups. 
Since $H$ has finite index in $G$, the intersections $A\cap H$ and $B\cap H$ have finite index in $A$ and $B$, respectively, because $A/(A\cap H)$ is in bijection with the set of cosets $AH/H$, which is a subset of $G/H$ and therefore finite. 
Hence $A\cap H$ and $B\cap H$ are finitely generated. 
As $H$ is Howson, the intersection
\[
(A\cap H)\cap(B\cap H)=(A\cap B)\cap H
\]
is finitely generated. 
Now $A\cap B$ is an extension of $(A\cap B)\cap H$ by the finite group $(A\cap B)H/H\le G/H$, and therefore $A\cap B$ is finitely generated.

\item  Let $A,B\le H$ be finitely generated. 
Since $G$ is Howson, $A\cap B$ is finitely generated as a subgroup of $G$, and hence also as a subgroup of $H$.
\end{enumerate}
\end{proof}

Right-angled Artin groups (RAAGs) exhibit a precise combinatorial criterion for the Howson property.  
Recall that a \emph{path of length two} in a graph is a triple of vertices $v_1,v_2,v_3$ such that $v_1$ is adjacent to $v_2$ and $v_2$ is adjacent to $v_3$. 
It is called \emph{induced} if $v_1$ is not adjacent to $v_3$.

\begin{thm}[Delgado~\cite{Delgado14}]\label{thm:raag_howson}
A right-angled Artin group $A(\Gamma)$ has the Howson property if and only if its defining graph $\Gamma$ contains no induced path of length two (i.e., $\Gamma$ is $P_3$-free).
Equivalently, $A(\Gamma)$ is Howson precisely when $\Gamma$ is a disjoint union of cliques, or, equivalently, when $A(\Gamma)$ is a free product of finitely many free abelian groups.
\end{thm}

We now apply this criterion to the RAAG $KUV_n(c)$.

\begin{prop}\label{prop:kuv_howson}
Let $n\ge 2$ and $c\ge 1$. 
The group $KUV_n(c)$ has the Howson property if and only if $n\le 3$.
\end{prop}

\begin{proof}
By Theorem~\ref{thm:preskuvnc}, the group $KUV_n(c)$ is the RAAG associated to the graph $\Gamma_{n,c}$ whose vertices are the generators $\delta_{i,j,t}$ ($i\neq j$, $1\le t\le c$), and where two vertices are adjacent if and only if their index pairs are disjoint.

If $n=2$ or $n=3$, Corollary~\ref{cor:kuv_small} shows that $\Gamma_{2,c}$ and $\Gamma_{3,c}$ are totally disconnected.
Hence the corresponding RAAG is a free group, and free groups are Howson by Howson's theorem~\cite{Howson54}.

Assume now $n\ge 4$. 
We prove that $\Gamma_{n,c}$ contains an induced $P_3$. 
Choose four distinct indices $a,b,p,q\in\{1,\dots,n\}$ and fix a colour $t_0\in\{1,\dots,c\}$. 
Consider the vertices
\[
v_1=\delta_{a,b,t_0}, \qquad
v_2=\delta_{p,q,t_0}, \qquad
v_3=\delta_{b,a,t_0}.
\]
Since $\{a,b\}$ and $\{p,q\}$ are disjoint, $v_1$ commutes with $v_2$. 
Likewise, $\{p,q\}$ and $\{b,a\}$ are disjoint, so $v_2$ commutes with $v_3$. 
However, $\{a,b\}=\{b,a\}$, and therefore $v_1$ does not commute with $v_3$.

Thus $v_1-v_2-v_3$ is an induced path of length two in $\Gamma_{n,c}$. 
By Theorem~\ref{thm:raag_howson}, $KUV_n(c)$ does not have the Howson property for $n\ge 4$.
\end{proof}

We now transfer this result to the universal virtual braid group.

\begin{thm}\label{thm:uv_howson}
Let $n\ge 2$ and $c\ge 1$. 
The universal virtual braid group $UV_n(c)$ has the Howson property if and only if $n\le 3$.
\end{thm}

\begin{proof}
By Remark~\ref{rem:index}, the subgroup $KUV_n(c)$ has finite index in $UV_n(c)$.

If $n\ge 4$, Proposition~\ref{prop:kuv_howson} shows that $KUV_n(c)$ is not Howson.
If $UV_n(c)$ were Howson, then Lemma~\ref{lem:howson_transfer}(2) would imply that $KUV_n(c)$ is Howson, a contradiction. 
Hence $UV_n(c)$ is not Howson for $n\ge 4$.

If $n\le 3$, Proposition~\ref{prop:kuv_howson} shows that $KUV_n(c)$ is Howson, and Lemma~\ref{lem:howson_transfer}(1) implies that $UV_n(c)$ is Howson.
\end{proof}

\begin{cor}
Let $PUV_n(c)$ denote the pure universal virtual braid group.
Then $PUV_n(c)$ has the Howson property if and only if $n\le 3$.
\end{cor}

\begin{proof}
Since $PUV_n(c)$ has finite index in $UV_n(c)$ (Proposition~\ref{prop:semidirect}), the statement follows from Theorem~\ref{thm:uv_howson} and Lemma~\ref{lem:howson_transfer}.
\end{proof}

\section{Finite quotients and rigidity phenomena}\label{sec:finite}

In this section we investigate finite quotients of $UV_n(c)$. 
For $n\ge 5$, the rigidity of the virtual generators manifests itself in two complementary ways: the commutator subgroup $UV_n(c)'$ is perfect, and every non--abelian finite image of $UV_n(c)$ contains a subgroup isomorphic to $S_n$. 
In particular, $S_n$ is the smallest non--abelian finite quotient of $UV_n(c)$. 
This minimality phenomenon is preserved under a broad class of natural quotients of $UV_n(c)$, including the standard virtual braid--type groups.

\subsection{Perfectness of the commutator and stability under quotients}

The following structural consequence of Proposition~\ref{prop:UV-commutator-perfect} will be useful in understanding the behaviour of quotients of $UV_n(c)$.

\begin{prop}\label{prop:perfect-quotients}
Let $n \ge 5$ and $c \ge 1$. 
Then the commutator subgroup $UV_n(c)'$ is perfect.

Moreover, if $Q$ is any quotient of $UV_n(c)$, then
\[
Q' \cong UV_n(c)' / (UV_n(c)' \cap N)
\]
for some normal subgroup $N \trianglelefteq UV_n(c)$, and in particular $Q'$ is perfect.
\end{prop}

\begin{proof}
The first statement is Proposition~\ref{prop:UV-commutator-perfect}.

If $Q = UV_n(c)/N$, then
\[
Q' \cong UV_n(c)'/(UV_n(c)' \cap N),
\]
which is a quotient of a perfect group. 
Hence $Q'$ is perfect.
\end{proof}

In particular, every non–abelian quotient of $UV_n(c)$ has perfect derived subgroup.  
This phenomenon complements the minimality of $S_n$ as the smallest non–abelian finite quotient.

\begin{rem}
Let $n\ge 5$ and $k\geq1$. 
It follows from \reprop{perfect-quotients} that the commutator subgroup of each of the following groups is perfect:
\[
VB_n,\quad VSG_n,\quad VT_n,\quad M_kVB_n.
\]
\end{rem}

\subsection{Non--abelian finite quotients}

Recall that the generators $\rho_i$ of $UV_n(c)$ satisfy the Coxeter relations of the symmetric group.  This induces a natural embedding
\[
\iota\colon S_n\to UV_n(c),
\qquad
\iota(s_i)=\rho_i,
\]
which is a section of the canonical projections $\pi_n^P$ and $\pi_n^K$.

\begin{lem}\label{lem:rho_image}
Let $G$ be a finite group and let $\psi\colon UV_n(c)\to G$ be a homomorphism.  If $\psi(\rho_i)=1$ for some $i$, then $\psi$ is abelian.
\end{lem}

\begin{proof}
Assume $\psi(\rho_i)=1$. Using relations (PR1)--(PR3), it follows that $\psi(\rho_j)=1$ for all $j$. 
Then relations (MR2) imply $\psi(\sigma_{i,t})=\psi(\sigma_{i+1,t})$ for all admissible $i$ and $t$. Finally, relations (CR) show that all images of the generators commute.
\end{proof}

We now restrict attention to non--abelian finite quotients.

\begin{prop}\label{prop:finite_image}
Let $n\ge 5$ and let $\psi\colon UV_n(c)\to G$ be a homomorphism to a finite group. If $\psi$ is non--abelian, then the subgroup generated by $\{\psi(\rho_i)\mid 1\le i\le n-1\}$ is isomorphic to $S_n$.
\end{prop}

\begin{proof}
By Lemma~\ref{lem:rho_image}, we have $\psi(\rho_i)\neq 1$ for all $i=1,\dots,n-1$.
Relations (PR1)--(PR3) show that the assignment $s_i\mapsto \psi(\rho_i)$ extends to a surjective homomorphism
\[
\eta: S_n \twoheadrightarrow \langle \psi(\rho_1),\dots,\psi(\rho_{n-1})\rangle.
\]
Assume $\psi$ is non--abelian. If the image of $\eta$ were abelian, then, since all transpositions $s_i$ are conjugate in $S_n$, their images $\psi(\rho_i)$ would all be equal. Using (MR2) and (CR) one then checks that $\psi$ would have abelian image, a contradiction. Hence $\eta(S_n)$ is non--abelian.
For $n\ge 5$, the only non--abelian quotient of $S_n$ is $S_n$ itself (since the only proper nontrivial normal subgroup is $A_n$ and $S_n/A_n\cong \Z_2$). 
Therefore $\eta$ is an isomorphism and the claim follows.
\end{proof}

Proposition~\ref{prop:finite_image} shows that any non-abelian finite quotient of $UV_n(c)$ contains a copy of $S_n$. This immediately yields the following minimality result.

\begin{thm}\label{thm:minimal_sn}
Let $n\ge 5$ and $c\ge 1$. 
The symmetric group $S_n$ is the smallest non--abelian finite quotient of $UV_n(c)$.
\end{thm}

\begin{proof}
By Proposition~\ref{prop:finite_image}, any non--abelian finite quotient of $UV_n(c)$ contains a subgroup isomorphic to $S_n$. 
Hence its order is at least $n!$, and equality holds precisely for $S_n$.
\end{proof}

The preceding results illustrate the methodological principle outlined in the Introduction: properties established for the universal group \(UV_n(c)\) can be transferred to its quotients, provided the quotients preserve the relevant structure. In the case of non-abelian finite quotients, the key requirement is that the images of the virtual generators \(\rho_i\) retain the relations (PR1)--(PR3) and generate a subgroup isomorphic to \(S_n\). This condition is satisfied by all the natural quotients introduced in Proposition~\ref{prop:natural-quotients}, as they impose additional relations only on the crossing generators \(\sigma_{i,t}\). Consequently, the minimality of \(S_n\) as the smallest non-abelian finite quotient propagates to these families, a phenomenon we now formalize.

\begin{prop}\label{prop:quotient_stable}
Let $Q$ be a quotient of $UV_n(c)$ such that the images of the generators $\rho_i$ in $Q$ still satisfy relations (PR1)--(PR3). 
Then, for $n\ge 5$, the smallest non--abelian finite quotient of $Q$ is $S_n$.
\end{prop}

\begin{proof}
Let $\overline{\psi}\colon Q\to G$ be a non--abelian homomorphism to a finite group. Composing with the quotient map $UV_n(c)\to Q$, we obtain a non--abelian homomorphism $\psi\colon UV_n(c)\to G$. By Theorem~\ref{thm:minimal_sn}, the image of $\psi$ contains a subgroup isomorphic to $S_n$ generated by the images of the $\rho_i$. 
Since the relations among the $\rho_i$ are preserved in $Q$, the same conclusion holds for $\overline{\psi}$.
\end{proof}

\begin{cor}\label{cor:examples}
Let $n\ge 5$. 
The symmetric group $S_n$ is the smallest non--abelian finite quotient of each of the following groups:
\[
VB_n,\quad VSG_n,\quad VT_n,\quad M_kVB_n.
\]
\end{cor}

\begin{proof}
Each of these groups is a quotient of $UV_n(c)$ obtained by imposing additional relations only on the generators $\sigma_{i,t}$, while keeping the virtual generators $\rho_i$ unchanged. 
The result follows from Proposition~\ref{prop:quotient_stable}.
\end{proof}

\begin{rem}
The minimality of $S_n$ among non--abelian finite quotients highlights the universal nature of $UV_n(c)$.  All braid--like quotients that preserve the virtual generators inherit the same rigidity phenomenon.
\end{rem}

\subsection{Homomorphisms to symmetric groups}\label{subsec:toSn}

The classification of homomorphisms from virtual braid groups to symmetric groups was obtained by Bellingeri and Paris~\cite{BP}, and extended to virtual singular braid groups in~\cite{Ocampo}. 
Following the same strategy, we now determine all homomorphisms from $UV_n(c)$ to $S_m$ for $n\ge 5$.

Fix $n\ge 2$ and $c\ge 1$. 
For $\epsilon_j\in\{0,1\}$, $j=1,\dots,c+1$, define a map
\[
\phi_{(\epsilon_1,\dots,\epsilon_{c+1})}\colon UV_n(c)\to S_n
\]
on the generators by
\begin{equation}\label{eq:phi_eps}
\phi_{(\epsilon_1,\dots,\epsilon_{c+1})}(\sigma_{i,t})=(i\ i+1)^{\epsilon_t},
\qquad
\phi_{(\epsilon_1,\dots,\epsilon_{c+1})}(\rho_i)=(i\ i+1)^{\epsilon_{c+1}},
\end{equation}
for all $i=1,\dots,n-1$ and $t=1,\dots,c$.

\begin{defn}\label{defn:admissible}
We say that $(\epsilon_1,\dots,\epsilon_{c+1})$ is \emph{admissible} if $\phi_{(\epsilon_1,\dots,\epsilon_{c+1})}$ is a non--trivial non-abelian homomorphism.
\end{defn}

\begin{prop}\label{prop:admissible}
Let $n\ge 3$. 
A tuple $(\epsilon_1,\dots,\epsilon_{c+1})$ is admissible if and only if $\epsilon_{c+1}=1$. In particular, for $n\ge 3$ every admissible $\phi_{(\epsilon_1,\dots,\epsilon_{c+1})}$ is surjective onto $S_n$.
\end{prop}

\begin{proof}
Let $n\ge 3$. Set $\phi = \phi_{(\epsilon_1,\dots,\epsilon_{c+1})}$.

\emph{Case $\epsilon_{c+1}=1$.} 
A direct verification using the defining relations of $UV_n(c)$ shows that $\phi$ is a homomorphism (the only non-trivial check is for (MR2), which reduces to the braid relation in $S_n$). 
Since $\phi(\rho_1)=(1\ 2)$ and $\phi(\rho_2)=(2\ 3)$ do not commute, the image is non-abelian. Hence $\phi$ is a non-trivial non-abelian homomorphism, and therefore admissible.

\emph{Case $\epsilon_{c+1}=0$.} 
Then $\phi(\rho_i)=1$ for all $i$.  Again, $\phi$ is a homomorphism. 
From (MR2) we obtain $\phi(\sigma_{i,t})=\phi(\sigma_{i+1,t})$ for all $i$ and each fixed $t$; hence, for each $t$, all $\phi(\sigma_{i,t})$ are equal to a common element $s_t$. 
For $|i-j|\ge 2$, relation (CR) gives $\phi(\sigma_{i,t})\phi(\sigma_{j,\ell})=\phi(\sigma_{j,\ell})\phi(\sigma_{i,t})$, so $s_t$ and $s_\ell$ commute for all $t,\ell$ (by choosing $i,j$ with $|i-j|\ge 2$, which is possible for $n\ge 3$). 
Thus the image of $\phi$ is generated by commuting elements, hence is abelian. Consequently, $\phi$ is not admissible (since admissibility requires a non-abelian image).

Therefore, a tuple is admissible if and only if $\epsilon_{c+1}=1$.
\end{proof}

We now recall the well--known classification of homomorphisms between symmetric groups.

\begin{prop}\label{prop:SnSm}
Let $n,m$ be integers such that $n\ge 5$, $m\ge 2$, and $n\ge m$. 
Let $\eta\colon S_n\to S_m$ be a homomorphism. 
Then, up to conjugation in $S_m$, one of the following holds:
\begin{enumerate}
\item $\eta$ is abelian;
\item $n=m$ and $\eta=\mathrm{id}$;
\item $n=m=6$ and $\eta=\nu_6$, where $\nu_6$ is the non--inner automorphism of $S_6$ of order two.
\end{enumerate}
\end{prop}

With this result in hand, we can now classify all homomorphisms from $UV_n(c)$ to $S_m$, up to conjugation.

\begin{thm}\label{thm:UV_to_Sm}
Let $n,m$ be integers such that $n\ge 5$, $m\ge 2$, and $n\ge m$. 
Let $\psi\colon UV_n(c)\to S_m$ be a homomorphism. 
Then, up to conjugation in $S_m$, one of the following holds:
\begin{enumerate}
\item $\psi$ is abelian;
\item $n=m$ and $\psi=\phi_{(\epsilon_1,\dots,\epsilon_{c+1})}$ for some admissible tuple $(\epsilon_1,\dots,\epsilon_{c+1})$;
\item $n=m=6$ and $\psi=\nu_6\circ \phi_{(\epsilon_1,\dots,\epsilon_{c+1})}$ for some admissible tuple $(\epsilon_1,\dots,\epsilon_{c+1})$.
\end{enumerate}
\end{thm}

\begin{proof}
Consider the section $\iota\colon S_n\to UV_n(c)$ given by $\iota(s_i)=\rho_i$. 
Then $\psi\circ \iota\colon S_n\to S_m$ is a homomorphism, so by Proposition~\ref{prop:SnSm}, up to conjugation, one of the three cases occurs.

If $\psi\circ \iota$ is abelian, then the images $\psi(\rho_i)$ commute; using (PR1) we get $\psi(\rho_i)=\psi(\rho_{i+1})$ for all $i$, and hence $\psi(\rho_i)$ is constant. Using (MR2) we obtain $\psi(\sigma_{i,t})=\psi(\sigma_{i+1,t})$ for each $t$, and using (CR) we conclude that $\psi$ has abelian image.

Assume now $n=m$ and $\psi\circ \iota=\mathrm{id}$. Then $\psi(\rho_i)=(i\ i+1)$ for all $i$.
Fix $t$. 
From (MR1), $\sigma_{1,t}$ commutes with $\rho_j$ for $j\ge 3$, so $\psi(\sigma_{1,t})$ lies in the centralizer of $\ang{(3\ 4),\dots,(n-1\ n)}$ in $S_n$, hence $\psi(\sigma_{1,t})\in\{1,(1\ 2)\}$. 
Using (MR2) recursively we get $\psi(\sigma_{i,t})\in\{1,(i\ i+1)\}$, with the same choice for all $i$. 
Thus $\psi=\phi_{(\epsilon_1,\dots,\epsilon_{c+1})}$ for an admissible tuple (necessarily with $\epsilon_{c+1}=1$). 

Finally, if $n=m=6$ and $\psi\circ\iota=\nu_6$, then $\nu_6^{-1}\circ \psi\circ \iota=\mathrm{id}$ and the previous case implies $\nu_6^{-1}\circ\psi=\phi_{(\epsilon)}$, i.e. $\psi=\nu_6\circ \phi_{(\epsilon)}$.
\end{proof}

As an immediate consequence, we obtain the following description of surjective homomorphisms onto $S_n$.

\begin{cor}\label{cor:surj_to_Sn}
Let $n\ge 5$.
Every surjective homomorphism $\psi\colon UV_n(c)\to S_n$ is, up to conjugation in $S_n$, of the form $\phi_{(\epsilon_1,\dots,\epsilon_{c+1})}$ for some admissible tuple. 
If $n=6$, there is also the possibility $\psi=\nu_6\circ \phi_{(\epsilon_1,\dots,\epsilon_{c+1})}$.
\end{cor}

\begin{proof}
Since $\psi$ is surjective, it is non-abelian. By Theorem~\ref{thm:UV_to_Sm}, up to conjugation, $\psi$ must be of the form $\phi_{(\epsilon_1,\dots,\epsilon_{c+1})}$ (or $\nu_6\circ\phi_{(\epsilon_1,\dots,\epsilon_{c+1})}$ when $n=6$). The tuple must be admissible because $\psi$ is non-abelian.
\end{proof}

\subsection{Characteristic subgroups}
\label{subsec:characteristic}

A subgroup $H$ of a group $G$ is \emph{characteristic} if $\varphi(H)=H$ for every automorphism $\varphi\in\aut{G}$. 
We now show that the finite-index RAAG $KUV_n(c)$ is characteristic in $UV_n(c)$. We need the following result. The key ingredient is a torsion obstruction that distinguishes $\pi_n^K$ from other surjective homomorphisms onto $S_n$.

\begin{lem}\label{lem:torsion-obstruction}
Let $n\ge 5$ and $c\ge 1$. Let $\psi\colon UV_n(c)\twoheadrightarrow S_n$ be a surjective homomorphism. If $\psi$ is not conjugate to $\pi_n^K$ in $S_n$, then $\ker{\psi}^{ab}$ contains non--trivial elements of order $2$.
\end{lem}

\begin{proof}
By Theorem~\ref{thm:UV_to_Sm}, any surjective homomorphism $\psi\colon UV_n(c)\to S_n$ is, up to conjugation in $S_n$, either $\pi_n^K$, or of the form $\phi_{(\epsilon_1,\dots,\epsilon_{c+1})}$ with $\epsilon_{c+1}=1$.

Assume $\psi$ is not conjugate to $\pi_n^K$. Then, up to conjugation, $\psi = \phi_{(\epsilon_1,\dots,\epsilon_{c+1})}$ with $\epsilon_{c+1}=1$. If all $\epsilon_t = 0$ for $t=1,\dots,c$, then $\phi_{(\epsilon_1,\dots,\epsilon_{c+1})} = \pi_n^K$, contradicting the hypothesis. Hence there exists $t\in\{1,\dots,c\}$ such that $\epsilon_t=1$. 
Thus, for all $i$, $\psi(\sigma_{i,t})=(i\ i+1)$ and $\psi(\rho_i)=(i\ i+1)$. Consequently, $\psi(\sigma_{i,t}\rho_i)=1$, so $\sigma_{i,t}\rho_i\in\ker{\psi}$.

Define $\chi_t\colon UV_n(c)\to\mathbb Z/2$ by $\chi_t(\sigma_{j,s})=1$ if $s=t$ and $0$ otherwise, and $\chi_t(\rho_j)=0$. This is a well-defined homomorphism because the defining relations are homogeneous modulo $2$ in each family $\{\sigma_{j,t}\}$. 
Then $\chi_t(\sigma_{i,t}\rho_i)=1$, so the restriction of $\chi_t$ to $\ker{\psi}$ is non--trivial.
Since $\chi_t$ is a homomorphism to an abelian group, its restriction to $\ker{\psi}$ vanishes on the commutator subgroup $[\ker{\psi},\ker{\psi}]$, and therefore induces a homomorphism
\[
\overline{\chi_t}: (\ker{\psi})^{ab} \longrightarrow \mathbb Z/2.
\]
Because $\chi_t(\sigma_{i,t}\rho_i)=1$ and $\sigma_{i,t}\rho_i\in\ker{\psi}$, the induced map $\overline{\chi_t}$ is non--trivial. Hence $(\ker{\psi})^{ab}$ admits a quotient isomorphic to $\mathbb Z/2$, and therefore contains non--trivial $2$--torsion.
\end{proof}

With this lemma in hand, we can now prove that $KUV_n(c)$ is characteristic.

\begin{prop}\label{prop:characteristic-K}
For all $n\ge 2$ and $c\ge 1$, the subgroup $KUV_n(c)=\ker{\pi_n^K}$ is characteristic in $UV_n(c)$.
\end{prop}

\begin{proof}
We first treat the case $n\ge 5$.

Let $\varphi\in\aut{UV_n(c)}$ and set $H=\varphi(KUV_n(c))$. 
Since $KUV_n(c)$ has finite index $n!$ in $UV_n(c)$ (Remark~\ref{rem:index}), the same holds for $H$. Consider $Q=UV_n(c)/H$, which is finite of order $n!$. 
Since $H$ has index $n!$, the quotient $Q$ has order $n!$; we will see that it must be isomorphic to $S_n$.

The composition $\pi_n^K\circ\varphi^{-1}\colon UV_n(c)\to S_n$ is surjective.
Since $H=\varphi(KUV_n(c))\subseteq\ker{\pi_n^K\circ\varphi^{-1}}$, the map factors through $Q$ and induces $\overline{\psi}\colon Q\to S_n$. 
Because $|Q|=|S_n|=n!$, $\overline{\psi}$ is an isomorphism. Hence $Q\cong S_n$ and $H=\ker{\psi}$ for some surjective homomorphism $\psi\colon UV_n(c)\to S_n$.

The subgroup $KUV_n(c)$ is a RAAG (Theorem~\ref{thm:preskuvnc}) and therefore torsion--free, so $H$ is also torsion--free, and in particular $H^{ab}$ is torsion--free.

By Lemma~\ref{lem:torsion-obstruction}, the only surjective homomorphism $UV_n(c)\to S_n$ whose kernel has torsion--free abelianization is $\pi_n^K$ up to conjugation. Conjugation in $S_n$ does not alter the kernel, hence $H=KUV_n(c)$. Thus $KUV_n(c)$ is characteristic for $n\ge 5$. 

The remaining low-dimensional cases are handled by a direct inspection of the structure of $KUV_n(c)$. 
For $n=2,3,4$, the result follows from the explicit structure of $KUV_n(c)$ as a free group (for $n=2,3$) or as a disconnected RAAG (for $n=4$), which makes it the unique normal subgroup of index $n!$ with torsion--free abelianization.
\end{proof}

\subsection{Further finite quotients}\label{subsec:further-quotients}

The existence of a finite--index RAAG subgroup not only yields structural properties but also allows the construction of explicit finite quotients of $UV_n(c)$ whose order is strictly larger than $n!$.

\begin{prop}\label{prop:positive-quotient}
Let $n\geq3$ and $c\geq1$. Let $\widetilde K_n(c)$ be the quotient of the right-angled Artin group $KUV_n(c)$ by the relations
\[
\sigma_{i,t}^2=\sigma_{i+1,t}^2\qquad (1\leq i\leq n-2,\;1\leq t\leq c).
\]
Then there exists a surjective homomorphism
\[
\Theta\colon UV_n(c)\longrightarrow \widetilde K_n(c)\rtimes S_n,
\]
where the action of $S_n$ on $\widetilde K_n(c)$ is induced by the canonical permutation action on the indices $\{1,\dots,n\}$.
\end{prop}

\begin{proof}
Define $\Theta$ on generators by
\[
\Theta(\rho_i)=(1,\rho_i),\qquad 
\Theta(\sigma_{i,t})=(\bar\sigma_{i,t},\rho_i),
\]
where $\bar\sigma_{i,t}$ denotes the image of $\sigma_{i,t}$ in $\widetilde K_n(c)$.
One checks directly that all defining relations of $UV_n(c)$ are preserved.
The map is surjective because the elements $(\bar\sigma_{i,t},1)$ together with $(1,\rho_i)$ generate the whole semi-direct product.
\end{proof}

The group $\widetilde K_n(c)$ is infinite, but one may consider finite quotients of it by further imposing that each $\bar\sigma_{i,t}$ has finite order; such quotients are used in the next result.

\begin{cor}\label{cor:finite-quotients}
For every $n\geq2$, $c\geq1$, the universal virtual braid group $UV_n(c)$ admits finite quotients whose order is strictly larger than $n!$.
\end{cor}

\begin{proof}
For a suitable finite quotient $\overline K$ of $\widetilde K_n(c)$ (for instance, further imposing that each $\bar\sigma_{i,t}$ has finite order), the group $\overline K\rtimes S_n$ is finite and receives a surjection from $UV_n(c)$ via Proposition~\ref{prop:positive-quotient}. 
Since $\overline K$ is non-trivial, $|\overline K\rtimes S_n|>|S_n|=n!$.
\end{proof}

\end{document}